\documentclass[11pt]{amsart}
\usepackage{url, calc}
\usepackage{amsmath,amsxtra,amssymb,latexsym,epsfig,amscd,amsthm,fancybox,epsfig}
\usepackage[mathscr]{eucal}
\usepackage{graphicx}
\usepackage{multicol,xcolor}
\usepackage{epsfig} 
\usepackage{epstopdf}
\usepackage{cases}
\usepackage{subfig}
\usepackage{color}
\usepackage{hyperref}
\usepackage{bigints}

\newtheorem{thm}{Theorem}[section]
\newtheorem {asp}{Assumption}[section]

\newtheorem{lm}{Lemma}[section]
\newtheorem{rmk}{Remark}[section]

\newtheorem{deff}{Definition}[section]

\newtheorem{prop}{Proposition}[section]
\theoremstyle{definition}

\theoremstyle{remark}

\newtheorem{example}{Example}[section]
\numberwithin{equation}{section}

\DeclareMathOperator{\suppo}{supp}
\DeclareMathOperator{\Conv}{Conv}
\newcommand{\eps}{\varepsilon}

\newcommand{\I}{\mathcal{I}}

\newcommand{\M}{\mathcal{M}}
\newcommand{\F}{\mathcal{F}}

\newcommand{\E}{\mathbb{E}}
\newcommand{\BE}{\mathbf{E}}
\newcommand{\BB}{\mathbf{B}}

\newcommand{\BX}{\mathbf{X}}
\newcommand{\bx}{\mathbf{x}}

\newcommand{\ba}{\mathbf{a}}

\newcommand{\bc}{\mathbf{c}}

\newcommand{\N}{\mathbb{N}}

\newcommand{\PP}{\mathbb{P}}

\newcommand{\R}{\mathbb{R}}

\newcommand{\U}{\mathcal{U}}

\newcommand{\wtd}{\widetilde}
\numberwithin{equation}{section}

\newcommand{\bed}{\begin{displaymath}}
\newcommand{\eed}{\end{displaymath}}
\newcommand{\bea}{\bed\begin{array}{rl}}
\newcommand{\eea}{\end{array}\eed}

\newcommand{\barray}{\begin{array}{ll}}
\newcommand{\earray}{\end{array}}

\newcommand{\1}{\boldsymbol{1}}
\newcommand{\0}{\boldsymbol{0}}
\newcommand{\bdelta}{\boldsymbol{\delta}}

\def\bar{\overline}
\def\hat{\widehat}
\def\a.s{\text{\;a.s.\;}}

\title[Persistence in food chains]{Persistence in Stochastic Lotka--Volterra food chains with intraspecific competition}
\author[A. Hening]{Alexandru Hening }
\address{Department of Mathematics\\
Tufts University\\
Bromfield-Pearson Hall\\
503 Boston Avenue\\
Medford, MA 02155\\
United States
}
\email{Alexandru.Hening@tufts.edu}

\author[D. Nguyen]{Dang H. Nguyen }
\thanks{D. Nguyen was in part supported by
 the National Science Foundation
under grant DMS-1207667.}
\address{Department of Mathematics \\
 Wayne State University\\
 Detroit, MI 48202 \\
 United States}
 \email{dangnh.maths@gmail.com}


\begin{document}
\maketitle

\begin{abstract}
This paper is devoted to the analysis of a simple Lotka-Volterra food chain evolving in a stochastic environment. It can be seen as the companion paper of Hening and Nguyen (J. of Math. Biol. `18) where we have characterized the persistence and extinction of such a food chain under the assumption that there is no intraspecific competition among predators. In the current paper we focus on the case when all the species experience intracompetition.
The food chain we analyze consists of one prey and $n-1$ predators. The $j$th predator eats the $j-1$st species and is eaten by the $j+1$st predator; this way each species only interacts with at most two other species - the ones that are immediately above or below it in the trophic chain.
We show that one can classify, based on the invasion rates of the predators (which we can determine from the interaction coefficients of the system via an algorithm), which species go extinct and which converge to their unique invariant probability measure. We obtain stronger results than in the case with no intraspecific competition because in this setting we can make use of the general results of Hening and Nguyen (Ann. of Appl. Probab.). Unlike most of the results available in the literature, we provide an in depth analysis for both non-degenerate and degenerate noise.

We exhibit our general results by analysing trophic cascades in a plant--herbivore--predator system and providing persistence/extinction criteria for food chains of length $n\leq 3$.
\end{abstract}
\tableofcontents

\section{Introduction}

Biological populations usually do not evolve in isolation and as such it is fundamentally important to determine which species persist and which go extinct in a given ecosystem. The fluctuations of the environment make the dynamics of populations inherently stochastic. Consequently, one needs to account for the combined effects of biotic interactions and environmental fluctuations when trying to determine species richness. Sometimes biotic effects can result in species going extinct. However, if one adds the effects of a random environment, extinction might be reversed into coexistence. In other instances, deterministic systems that coexist become extinct once one takes into account environmental fluctuations. A  successful method for studying this interplay is to model the populations as discrete or continuous time Markov processes and study the long-term behavior of these processes (\cite{C00, ERSS13, EHS15,  LES03, SLS09, SBA11, BEM07, BS09, BHS08, CM10, CCA09}).

Even though ecological systems are often more complex than a linear food chain, understanding food chain dynamics has a very long theoretical and empirical history and is a topic that is covered extensively in introductory biology (\cite{O71,W93, P92, H92, OPO95, VZ99, PCG00}). In nature the way species interact changes at least seasonally and is extremely complicated. Food chains are simplified `caricatures' of the real world but still offer interesting information about various biological features. Most of the time it is more realistic to model a system as a \textit{food web}, consisting of an interconnection of food chains. However, in certain instances the model can be simplified to a single food chain and one can glean relevant information by analyzing the properties of the linear food chain analytically. For example, if one has a system with three species in which one is a plant, the second one is a herbivore and the third is a predator then there is no reason to have a direct link between the plant and the predator -- in this setting one would get a linear food chain (see \cite{MR12}).

All ecologists are aware that the world is complex. There is a divide between those who are skeptical of theories based on simplified models (see, for example, \cite{P91}) and those who think that `simple models can be used like a surgeon's knife, cutting deftly through the cloying of fat of complicating detail to get at the essential sinews of ecological reality' (\cite{T10}). Both from practical and analytical perspectives, theoretical models in ecology must significantly simplify the natural complexities. For food webs one of the main simplifications is the community module, that is represented by a food chain. One of the simplest models of food chains is the Lotka--Volterra one. Even though this model is imperfect and does not describe the behavior of any actual ecosystem, it nevertheless captures some key features which carry over to more realistic and analytically intractable models.

In this paper we study models of food chains of arbitrary length. We assume that there is only one species at each trophic level and that each species eats only the one on the adjacent lower trophic level. Furthermore, the ecosystem is supposed to have no immigration or emigration.

Many of the food-chain models studied in the literature are deterministic and of Lotka-Volterra type. Criteria for persistence and extinction have been studied by \cite{GH79, G80, FS85} while the global stability of nonnegative equilibrium points was studied by \cite{S79, H79}.

Usually, individuals of the same species have similar requirements for survival. Sometimes their combined demand for a resource is higher than the supply. As such the individuals have to compete for the resource and during this competition some of the individuals naturally become deprived and are therefore less likely to survive. Ecologists usually term this competition among individuals of the same species as \textit{intraspecific} competition. Arguably all populations experience some form of intraspecific competition. In particular, all vertebrate top predators in terrestrial ecosystems, with the possible exception of some reptiles, have strong intraspecific competition due to direct aggresion or territoriality (\cite{T10}).

In most situations, competing individuals do not interact with one another directly. One such situation occurs due to \textit{exploitation} where individuals are affected by the resource that is left after it has been consumed by others. Another form of indirect competition is due to \textit{interference}. This happens when one individual will prevent another from exploiting the resource within an area of the habitat. Yet another possible type of interaction between individuals of the same species is \textit{intraspecific predation}. This is the process of both killing and eating an individual of the same species. These observations suggest that one cannot always ignore intraspecific competition, as has been done in previous work.

One example of a deterministic Lotka-Volterra food chain is given by the system:
\begin{equation}\label{e:det}
\begin{split}
dx_1(t) &= x_1(t)(a_{10} - a_{11}x_1(t) - a_{12}x_2(t))\,dt\\
dx_2(t) &= x_2(t)(-a_{20} + a_{21}x_1(t) - a_{22}x_2(t)- a_{23}x_3(t))\,dt\\
&\mathrel{\makebox[\widthof{=}]{\vdots}} \\
dx_{n-1}(t) &= x_{n-1}(t)(-a_{n-1,0}+a_{n-1,n-2}x_{n-2}(t) -a_{n-1,n-1}x_{n-1}(t) - a_{n-1,n}x_n)\,dt\\
dx_n(t) &= x_n(t)(-a_{n0} + a_{n,n-1}x_{n-1}(t)- a_{n,n}x_n(t))\,dt.
\end{split}
\end{equation}

The quantities $(x_1(t),\dots,x_n(t))$ represent the densities of the $n$ species at time $t\geq 0$.
In this model $x_1$ describes a prey species, which is at the bottom of the food chain. The next $n-1$ species are predators. Species $1$ has a per-capita growth rate $a_{10}>0$ and its members compete for resources according to the intracompetition rate $a_{11}>0$. Predator species $j$ has a death rate $-a_{j0}<0$, preys upon species $j-1$ at rate $a_{j,j-1}>0$, competes with its own members at rate $a_{jj}> 0$ and is preyed upon by predator $j+1$ at rate $a_{j,j+1}>0$. The last species, $x_n$, is considered to be the apex predator of the food chain.

In the deterministic setting one says that the system \eqref{e:det} is \textit{persistent} if each solution of $\bx(t) = (x_1(t),\dots,x_n(t))$ with $\bx(0)\in \R_+^{n,\circ}:=((y_1,\dots,y_n): y_i>0, i=1,\dots,n)$ satisfies
\[
\limsup_{t\to\infty} x_i(t) >0, i=1,\dots,n.
\]
We say that species $i$ goes \textit{extinct} if
\[
\lim_{t\to \infty} x_i(t) = 0.
\]
It is natural to
analyze the coexistence of species by looking at the average per-capita growth rate of a
popualtion when it is rare. Intuitively, if this growth (or invasion) rate is positive, the respective
population increases when rare and can invade, while if the growth is negative, the population
decreases and goes extinct. If there are only two populations, coexistence is ensured if each population can invade when it is rare and the other population is stationary (\cite{T77, CE89, EHS15}).

There is a general theory for coexistence for deterministic models (\cite{H81, H84, HJ89}). It is shown that a sufficient condition for persistence is the existence of a fixed set of weights associated with the interacting populations such that this weighted combination of the populations's invasion rates is positive for any invariant measure supported by the boundary (i.e. associated to a sub-collection of populations) - see \cite{H81}.

In order to take into account environmental fluctuations and their effect on the persistence or extinction of species, one approach is to study systems that have random environmental perturbations. One way to do this is by analysing stochastic differential equations that arise by adding noise to ordinary differential equations. For \textit{compact} state spaces there are results for persistence in \cite{SBA11}. These results have been generalized in \cite{HN16} where the authors show how, under some natural assumptions, one can characterize the coexistence and extinction of species living on \textit{non-compact} state spaces. Some of these results hold not only for stochastic differential equations but also for stochastic difference equations (see \cite{SBA11}), piecewise deterministic Markov processes (see \cite{BL16, HS17}), and for general Markov processes (see \cite{B14}).

One stochastic version of \eqref{e:det} is the process $\BX:=(\BX(t))_{t\geq 0}=(X_1(t),\dots,X_n(t))_{t\geq 0}$ defined by the system of stochastic differential equations

\begin{equation}\label{e:stoc}
\begin{split}
dX_1(t) &= X_1(t)(a_{10} - a_{11}X_1(t) - a_{12}X_2(t))\,dt + X_1(t)\,dE_1(t)\\
dX_2(t) &= X_2(t)(-a_{20} + a_{21}X_1(t) - a_{22}X_2(t) - a_{23}X_3(t))\,dt+X_2(t)\,dE_2(t)\\
&\mathrel{\makebox[\widthof{=}]{\vdots}} \\
dX_{n-1}(t) &= X_{n-1}(t)(-a_{n-1,0}+a_{n-1,n-2}X_{n-2}(t) -a_{n-1,n-1}X_{n-1}(t)- a_{n-1,n}X_n)\,dt\\
&~~~~~+X_{n-1}(t)\,dE_{n-1}(t)\\
dX_n(t) &= X_n(t)(-a_{n0} + a_{n,n-1}X_{n-1}(t)-a_{nn}X_n(t))\,dt + X_n(t)\,dE_n(t)
\end{split}
\end{equation}
where $\BE(t)=(E_1(t),\dots, E_n(t))^T=\Gamma^\top\BB(t)$ for an $n\times n$ matrix
$\Gamma$ such that
$\Gamma^\top\Gamma=\Sigma=(\sigma_{ij})_{n\times n}$
and $\BB(t)=(B_1(t),\dots, B_n(t))$ is a vector of independent standard Brownian motions living on the probability space $(\Omega,\F,\{\F_t\}_{t\geq0},\PP)$ with a filtration $\{\F_t\}_{t\geq 0}$ satisfying the usual conditions. We denote by $\PP_\bx$ (respectively $\E_\bx$) the probability measure (respectively the expected value) conditioned on $\BX(0)=(X_1(0),\dots,X_n(0))=\bx\in\R_+^n$.

\begin{rmk}
There are a few different ways to add stochastic noise to deterministic population dynamics. We assume that the environment mainly affects the growth/death rates of the populations. This way, the growth/death rates in an ODE (ordinary differential equation) model are replaced by their average values plus random noise fluctuation terms. See \cite{T77, B02, G88, HNY16, EHS15, ERSS13, SBA11, HN16, G84} for more details.
\end{rmk}
Define the stochastic growth rate $\tilde a_{10} := a_{10}-\frac{\sigma_{11}}{2}$ and the stochastic death rates $\tilde a_{j0} := a_{j0} + \frac{\sigma_{jj}}{2}, j=1,\dots,n$. For fixed $j\in \{1,\dots,n\}$ write down the system
\begin{equation}\label{e:system}
\begin{split}
-a_{11}x_1 - a_{12}x_2 &= -\tilde a_{10}\\
a_{21}x_1 - a_{22}x_2 - a_{23}x_3 &=\tilde a_{20}\\
&\mathrel{\makebox[\widthof{=}]{\vdots}} \\
a_{j-1,j-2}x_{j-2} - a_{j-1,j-1}x_{j-1} -a_{j-1,j}x_{j} &= \tilde a_{j-1,0}\\
a_{j,j-1}x_{j-1} - a_{jj}x_j &= \tilde a_{j0}.
\end{split}
\end{equation}
It is easy to show that \eqref{e:system} has a unique solution, say $(x^{(j)}_1,\dots,x^{(j)}_j)$.
Define
\begin{equation}\label{e:inv_j+1}
\I_{j+1} = -\tilde a_{j+1,0} + a_{j+1,j} x^{(j)}_j.
\end{equation}
We will show that, if \eqref{e:system} has a strictly positive solution $(x^{(j)}_1,\dots,x^{(j)}_j)$, the invasion rate of predator $X_{j+1}$ in the habitat of $(X_1,\dots,X_j)$ is given by \eqref{e:inv_j+1}.
The invasion rate of predator $X_{j+1}$ is the asymptotic logarithmic growth $\lim_{t\to\infty}\frac{\log X_{j+1}(t)}{t}$ when $X_{j+1}$ is introduced at a low density in $(X_1,\dots,X_j)$.
We also set $\I_1:=\tilde a_{10}$ to be the stochastic growth rate of the prey - this can be seen as the invasion rate of the prey into the habitat, when it is introduced at low densities.

Throughout the paper we define $\R^n_+=[0,\infty)^n$ and
for $j=1,\dots,n$ $$\R_+^{(j)}:=\{\bx=(x_1,\dots,x_n)\in\R^n_+: x_{k}=0\,\text{ for } j<k\leq n\} \subset \R_+^n,$$
and
$$\R_+^{(j),\circ}:=\{\bx=(x_1,\dots,x_n)\in\R^n_+: x_k>0\,\text{ for } k\leq j; x_{k}=0\,\text{ for } j<k\leq n\}.$$

\begin{deff}
One can define a distance on the space of probability measures living on the space $(\R_+^n,\mathcal{B}(\R_+^n))$ i.e. the Borel measurable subsets of $\R_+^n$. This is done by defining $\|\cdot,\cdot\|_{\text{TV}}$, the \textit {total variation norm}, via
\[
\|\mu,\nu\|_{\text{TV}} := \sup_{A\in \mathcal{B}(\R_+^n)} |\mu(A)-\nu(A)|.
\]
\end{deff}

There are different ways one can define the persistence and extinction of species. We review some of these definitions below.

For a system to be strongly stochastically persistent we require that there exists a unique invariant measure $\pi^*$ that does not put any mass on the extinction set $S:=\{\bx\in\R_+^n: \Pi_{i=1}^nx_i=0\}$ and that the distribution of $\BX$ converges in some sense to $\pi^*$.

\begin{deff}
The process $\BX$ is \textbf{strongly stochastically persistent} if it has a unique invariant probability measure $\pi^*$ on $\R^{n,\circ}_+$ and converges weakly to $\pi^*$, that is
\[
P_\BX(t,\mathbf{x},\cdot)\Rightarrow \pi^*, ~\text{as}~t\to\infty, ~\ \bx\in\R_+^{n,\circ}
\]
where $P_\BX(t,\mathbf{x},\cdot)$ is the transition probability of $\BX$. This means that for any continuous function
  $f:\R^n_+\mapsto\R$ with $\sup_{\bx\in\R^n_+}|f(\bx)|\leq 1$ and any $\bx_0\in\R^{n,\circ}_+$
\begin{equation*}\label{e19-thm1.1}
\lim_{t\to\infty}\E_{\bx_0}f(\BX(t)) =  \int_{\R^n_+} f(\bx')\pi_{j^*}(d\bx').
\end{equation*}
\end{deff}
\begin{rmk}
We note that if
\begin{equation*}
\lim\limits_{t\to\infty} \|P_\BX(t, \mathbf{x}, \cdot)-\pi^*(\cdot)\|_{\text{TV}}=0,
\end{equation*}
then
\[
P_\BX(t,\mathbf{x},\cdot)\Rightarrow \pi^*,  ~\text{as}~t\to\infty
\]
so that convergence in total variation implies weak convergence.
\end{rmk}
\begin{deff}
The species $X_i$ goes \textbf{extinct} if for all $\bx\in\R^{n,\circ}_+$
\[
\PP_\bx\left\{\lim_{t\to\infty}X_i(t)=0\right\}=1.
\]
\end{deff}
\begin{deff}
The species $(X_1,\dots,X_{j^*})$ are \textbf{ persistent in probability} if
for any $\eps>0$, there exists a compact set $K_\eps\subset\R^{(j^*),\circ}_+$ such that
\[
\liminf_{t\to\infty}\PP_\bx\left\{(X_1(s),\dots, X_k(s))\in K_\eps\right\}\,ds \geq 1-\eps, \text{ for any }\,\bx\in\R^{n,\circ}_+,
\]
where $\left(x^{(j^*)}_1,\dots,x^{(j^*)}_{j^*}\right)\in\R_+^{(j^*),\circ} $is the unique solution to \eqref{e:system} with $j=j^*$
\end{deff}
We refer the reader to \cite{S12} for a discussion of various forms of persistence. With the above concepts in hand we can formulate our main result.
\begin{thm}\label{t:main}
Suppose $n\geq 2$, and $\BX(0)=\bx\in\R_+^{n,\circ}$. We have the following classification.
\begin{itemize}
\item [(i)] If $\I_n>0$ then $(X_1,\dots,X_n)$ is persistent in probability.
Moreover,
\begin{equation}\label{e0a-thm1.1}
\PP_\bx\left\{\lim_{t\to\infty}\dfrac1t\int_0^t X_k(s)\,ds = x^{(n)}_k>0, k=1,\dots,n\right\}=1
\end{equation}
where $\left(x^{(n)}_1,\dots,x^{(n)}_{n}\right)\in\R_+^{(n),\circ} $is the unique solution of \eqref{e:system} with $j=n$.

If $\Sigma$ is positive definite, making the noise non-degenerate, then the food chain $\BX$
is strongly stochastically persistent and its transition probability converges to its unique invariant probability measure $\pi^{(n)}$ on $\R_+^{n,\circ}$ exponentially fast in total variation.
\item [(ii)] If there exists $0\leq j^*<n$ such that $\I_{j^*}>0$ and $\I_{j^*+1}<0$ then $X_{j^*+1},\dots,X_n$ go extinct almost surely exponentially fast, as $t\to\infty$, with rates  $\I_{j^*+1},-\tilde{a}_{j^*+2,0},\dots,-\tilde{a}_{n0}$ respectively.
Furthermore, $(X_1,\dots,X_{j^*})$ is persistent in probability and with probability $1$
$$
\lim_{t\to\infty}\dfrac1t\int_0^t X_i(s)ds=
\begin{cases}
x^{(j^*)}_i\,&\text{ if } i=1,\dots,j^*,\\
0\,&\text{ if } i=j^*+1,\dots, n.
\end{cases}
$$
where $\left(x^{(j^*)}_1,\dots,x^{(j^*)}_{j^*}\right)\in\R_+^{(j^*),\circ} $is the unique solution of \eqref{e:system} with $j=j^*$.

\item [(iii)] Suppose that $\I_{j^*}>0$ and $\I_{j^*+1}<0$ for some $j^*<n$.
Suppose further that there exists a unique invariant probability measure $\pi_{j^*}$
on $\R^{(j^*),\circ}_+$ such that the transition probability measure
of $\BX$ restricted on $\R^{(j^*),\circ}_+$ converges weakly uniformly on each compact set to $\pi_{j^*}$. By this we mean that for any continuous and bounded function: $f:\R^{(j^*),\circ}_+\mapsto\R_+$
and for any compact set $K\subset \R^{(j^*),\circ}_+$, we have
\begin{equation}\label{uwc}
\lim_{t\to\infty}\left(\sup_{\bx\in K}\left|\int_{\R^{(j^*),\circ}_+} f(\bx')\pi_{j^*}(d\bx')-\E_{\bx}f(\BX(t))\right|\right)=0.
\end{equation}
Then for any $\bx\in\R^{n,\circ}_+$,
the transition probability measure $P(t,\bx,\cdot)$
of $\BX$ converges weakly to $\pi_{j^*}$ and as a result $(X_1,\dots,X_{j^*})$ is strongly stochastically persistent.

\item [(iv)]Suppose that $\I_{j^*}>0, \I_{j^*+1}<0$ for some $j^*<n$ and $\Sigma_{j^*}$, the principal submatrix of $\Sigma$ obtained by removing the $j^*+1$-th,\dots, $n$-th rows and columns of $\Sigma$, is positive definite.  Then for any $\bx\in\R^{n,\circ}_+$,
the transition probability measure $P(t,\bx,\cdot)$ of $\BX$ converges weakly to $\pi_{j^*}$  and as a result $(X_1,\dots,X_{j^*})$ is strongly stochastically persistent.
\end{itemize}
\end{thm}
\begin{rmk}
If $\Sigma$ is positive definite then all principal submatrices are positive definite, so in particular $\Sigma_{j^*}$ from Theorem \ref{t:main} part (iv) is positive definite.
\end{rmk}
\begin{rmk}\label{r:extra_predator}
We note that by Theorem \ref{t:main} the food chain persists when $\I_n>0$ and goes extinct when $I_{j^*+1}<1$ for some $j^*\leq n-1$.  It is key to note that $\I_j$ is independent of the coefficients $(a_{lm}), l>j$.

As such, if we add one extra predator at the top of the food chain the quantities  $\I_j>0, j=2,\dots,n$ remain unchanged and we get one extra invasion rate $\I_{n+1}$.
In this setting, when we have $n$ predators, the system persists if $\I_{n+1}>0$ and goes extinct if  $I_{j^*+1}<1$ for some $j^*\leq n$. This means that the introduction of an apex predator makes extinction \textbf{more likely}.
\end{rmk}

\begin{rmk}
The persistence or extinction of species evolving according to system \eqref{e:det} when the intraspecies competition for predators is zero (i.e. $a_{ii}=0, i\geq 2$) has been studied by \cite{GH79}. \cite{HN17} generalized the results from \cite{GH79} to a stochastic setting. The current paper tackles the case when intraspecies competition is nonzero. We get stronger results than in the case without intracompetition because we are able to make use of the general results from \cite{HN16}. From a technical point of view, strictly positive intracompetition rates make the process return to compact sets exponentially fast. This fact can then be used to prove exponential convergence to an invariant probability measure or extinction.
\end{rmk}

Most of the results for stochastic food chains only consider chains of length two. We note that our results are new even in the case of food chains of length three.

Theorem \ref{t:main} extends previous results on stochastic Lotka-Volterra systems in two dimensions (see \cite{LB16, HN16, R03}) to an $n$ dimensional setting. We also generalize the work by \cite{G84} where the author gives sufficient conditions for persistence of stochastic Lotka-Volterra type food web models in bounded regions of state space. We note that the main results of \cite{G84} only say something about persistence until the first exit time of the process from a compact rectangular region $R_\gamma\subset \R_+^{n,\circ}$. Once the process exits the region, one cannot say whether the species persist or not. Partial results for the existence of invariant probability measures for stochastic Lotka-Volterra systems have been found in \cite{P79}. However, these conditions are quite restrictive and impose artificial constraints on the interaction coefficients. In contrast, our results for persistence and extinction are sufficient and (almost) necessary. Moreover, based on which conditions are satisfied, we can say exactly which species persist and which go extinct.

The paper is organized as follows. In Section \ref{s:math} we present the mathematical framework from \cite{HN16} and explain how we can apply it in the current context. The proof of Theorem \ref{t:main} is presented in Appendix \ref{s:proofs}. General properties regarding the invasion rates and algorithms for how one can compute these invasion rates appear in Section \ref{s:inv}. In Section \ref{s:cascade} we study a plant--herbivore--predaftor food chain and look at the \textit{trophic cascade} effect the predator has. Finally, Section \ref{s:disc} is devoted to discussing our results and comparing them to the literature.

\section{Mathematical framework}\label{s:math}
We rewrite \eqref{e:stoc} as

\begin{equation}\label{e:system_2}
dX_i(t)=X_i(t) f_i(\BX(t))dt+X_i(t)dE_i(t), ~i=1,\dots,n
\end{equation}
where $\BX(t):=(X_1(t),\dots,X_n(t))$. This is a stochastic process that takes values in $\R_+^{n}:=[0,\infty)^n$ and defined on a complete probability space $(\Omega,\F,\{\F_t\}_{t\geq0},\PP)$ with a filtration $\{\F_t\}_{t\geq 0}$ satisfying the usual conditions.
We mainly focus on the process $\BX$ starting at $\bx\in\R^{n,\circ}_+=(0,\infty)^n$.
The random normalized occupation measures are defined as $$\wtd \Pi_t(B):=\dfrac1t\int_0^t\1_{\{\BX(s)\in\cdot\}}ds,\,t>0, B\in\mathcal{B}(\R_+^{n})$$
where $\mathcal{B}(\R_+^{n})$ are the Borel measurable subsets of $\R_+^{n}$. Note that $\wtd \Pi_t(B)$ tells us the fraction of time the process $\BX$ spends in the set $B$ during the duration $[0,t]$.

Let $\M$ be the set of ergodic invariant probability measures of $\BX$ supported on the boundary $\partial\R^n_+:=\R_+^n\setminus \R_+^{n,\circ}$. For a subset $\wtd\M\subset \M$, denote by $\Conv(\wtd\M)$ the convex hull of $\wtd\M$,
that is the set of probability measures $\pi$ of the form
$\pi(\cdot)=\sum_{\mu\in\wtd\M}p_\mu\mu(\cdot)$
with $p_\mu>0,\sum_{\mu\in\wtd\M}p_\mu=1$.

Note that each subspace of $\R^n_+$ of the form
$$\Big\{(x_1,\dots,x_n)\in\R^n_+: x_i>0 \text{ for } i\in\{\tilde n_1,\dots,\tilde n_k\}; \text{ and }x_i=0\text{ if } i\notin\{\tilde n_1,\dots,\tilde n_k\} \Big\}$$
for some $\tilde n_1,\dots,\tilde n_k\in\N$
satisfying $0<\tilde n_1<\dots<\tilde n_k\leq n$
is an invariant set for the process $\BX$.
Thus, any ergodic measure $\mu\in\M$
must be supported in such a subspace, that is,
there exist $0<n_1<\dots< n_k\leq n$
(if $k=0$, there are no $n_1,\dots, n_k$)
such that $\mu(\R^{\mu,\circ}_+)=1$ where
$$\R_+^\mu:=\{(x_1,\dots,x_n)\in\R^n_+: x_i=0\text{ if } i\in I_\mu^c\}$$
for
$I_\mu:=\{n_1,\dots, n_k\}$,
$I_\mu^c:=\{1,\dots,n\}\setminus\{n_1,\dots, n_k\}$,
$$\R_+^{\mu,\circ}:=\{(x_1,\dots,x_n)\in\R^n_+: x_i=0\text{ if } i\in I_\mu^c\text{ and }x_i>0\text{ if  }x_i\in I_\mu\},$$ and $\partial\R_+^{\mu}:=\R_+^n\setminus\R_+^{\mu,\circ}$. For the Dirac measure $\bdelta^*$ concentrated at the origin $0$, we have $I_{\bdelta^*}=\emptyset$

\begin{rmk}
Note that $\Conv(\M)$ is exactly the set of invariant probability measures of the process $\BX$ supported on the boundary $\partial \R_+^{n}$.
\end{rmk}

For a probability measure  $\mu$ on $\R^n_+$ we define the $i$th Lyapunov exponent (when it exists) via
\begin{equation}\label{Lya.exp}
\begin{aligned}
\lambda_j(\mu):=&\int_{\R^n_+}\left(f_j(\bx)-\dfrac{\sigma_{jj}}2\right)\mu(d\bx)\\
=&
\begin{cases}
\int_{\R^n_+}\left(\tilde a_{10}-a_{11}x_1 - a_{12}x_2\right)\mu(d\bx) &\text{ if }\, j=1,\\
\int_{\R^n_+}\left(-\tilde a_{n0}+a_{n,n-1}x_{n-1}-a_{n,n}x_{n}\right)\mu(d\bx)& \text{ if } j=n,\\
\int_{\R^n_+}\left(-\tilde a_{j,0}+a_{j,j-1}x_{j-1}-a_{j,j}x_{j}  -a_{j,j+1}x_{j+1}\right)\mu(d\bx)& \text{ otherwise}.
\end{cases}
\end{aligned}
\end{equation}

\begin{rmk}
To determine the Lyapunov exponents of an ergodic invariant probability measure $\mu\in\M$,
one can look at the equation for $\ln X_i(t)$. An application of It\^o's Lemma yields that		
$$
\dfrac{\ln X_i(t)}t=\dfrac{\ln X_i(0)}t+\dfrac1t\int_0^t\left[f_i(\BX(s))-\dfrac{\sigma_{ii}}2\right]ds+\dfrac1t\int_0^t dE_i(s).
$$

If $\BX$ is close to the support of an ergodic invariant measure $\mu$ for a long time $t\gg 1$,
then
$$\dfrac1t\int_0^t\left[f_i(\BX(s))-\dfrac{\sigma_{ii}}2\right]ds$$
can be approximated by the average with respect to $\mu$
$$\lambda_i(\mu)=\int_{\partial \R^n_+}\left(f_i(\bx)-\dfrac{\sigma_{ii}}2\right)\mu(d\bx).$$
On the other hand, the term $$\dfrac{\ln X_i(0)}t+\frac{E_i(t)}{ t}$$ is negligible for large $t$ since
\[
\PP_\bx \left\{\lim_{t\to\infty} \left(\dfrac{\ln X_i(0)}t+\frac{E_i(t)}{ t}\right)=0\right\}=1.
\]
This implies that $\lambda_i(\mu), i=1,\dots, n$ are the Lyapunov exponents of $\mu$.
\end{rmk}
For $\bx=(x_1,\dots,x_n)\in\R^n$,
we define  the norm $\|\bx\|=\max_{i=1}^n|x_i|$.
Let
\begin{equation}\label{e.c}
\bc=(c_1,\dots,c_n)\in\R^{n,\circ}_+,\text{ where }\, c_1=1, c_i:=\prod_{j=2}^i\dfrac{a_{k-1,k}}{a_{k,k-1}}, i\geq 2.
\end{equation}
 One can easily check that there exists  $\gamma_b>0$ such that
\begin{equation}\label{a.tight}
\limsup\limits_{\|x\|\to\infty}\left[\dfrac{\sum_i c_ix_if_i(\bx)}{1+\sum_i c_ix_i}-\dfrac12\dfrac{\sum_{i,j} \sigma_{ij}c_ic_jx_ix_j}{(1+\sum_i c_ix_i)^2}+\gamma_b\left(1+\sum_{i} (|f_i(\bx)|)\right)\right]<0.
\end{equation}
Then parts (2) and (3) of Assumption 1.1 in \cite{HN16} are satisfied and one gets the existence and uniqueness of strong solutions to \eqref{e:system_2}. Moreover, if $\BX(0)=\bx \in \R_+^{n,\circ}$ then
\[
\PP_\bx\{\BX(t)\in \R_+^{n,\circ}, t\geq 0\} = 1.
\]
In view of \cite[Lemma 2.3]{HN16}, for $\mu\in\M$, $\lambda_i(\mu)$ is well-defined and
\begin{equation}\label{e:lambda_0}
\lambda_i(\mu) = 0, i\in I_\mu.
\end{equation}
The intuition behind equation \eqref{e:lambda_0} is the following: if we are inside the support of an ergodic invariant measure $\mu$ then we are at an `equilibrium' and the process does not tend to grow or decay.
If  $\mu$ is an invariant probability measure satisfying $\mu(\R_+^{(j),\circ})=1$ then
we derive from \eqref{e:lambda_0} that
\begin{equation}\label{e.meanmu}
\E_{\mu}X_i=\int_{\R^n_+} x_i\mu(d\bx)=x_i^{(j)} \,\text{ for } i\leq j.
\end{equation}
That is, the solution of \eqref{e:system} is the vector $(\E_{\mu} X_1,\dots, \E_{\mu} X_j)$ of the expected values of $(X_1,\dots,X_j)$ at stationarity.

The following assumption is shown in \cite{HN16} to imply strong stochastic persistence

\begin{asp}\label{a.coexn}
For any $\mu\in\Conv(\M)$ one has
$$\max_{\{i=1,\dots,n\}}\left\{\lambda_i(\mu)\right\}>0.$$
\end{asp}

Extinction is ensured by the following two assumptions.
\begin{asp}\label{a.extn}
There exists $\mu\in\M$ such that
\begin{equation}\label{ae3.1}
\max_{i\in I_\mu^c}\{\lambda_i(\mu)\}<0.
\end{equation}
If $\\R_+^n\ne\{\0\}$, suppose further that
for any $\nu\in\Conv(\M_\mu)$ , we have
\begin{equation}\label{ae3.2}
\max_{i\in I_\mu}\{\lambda_i(\nu)\}>0
\end{equation}
where $\M_\mu:=\{\nu'\in\M:\suppo(\nu')\subset\partial\R_+^n\}.$
\end{asp}
Define
\begin{equation}\label{e:M1}
\M^1:=\left\{\mu\in\M : \mu ~~\text{satisfies Assumption} ~~\ref{a.extn}\right\}
\end{equation}
and
\begin{equation}\label{e:M2}
\M^2:=\M\setminus\M^1.
\end{equation}

\begin{asp}\label{a.extn3}
Suppose that one of the following is true
\begin{itemize}
  \item $\M^2=\emptyset$

  \item For any $\nu\in\Conv(\M^2)$, $\max_{\{i=1,\dots,n\}}\left\{\lambda_i(\nu)\right\}>0.$
\end{itemize}
\end{asp}

\begin{rmk}
We refer the reader to \cite{HN16} for a detailed discussion of the above assumptions. In short
\begin{itemize}
\item From a dynamical point of view, the solution in the interior domain $\R^{n,\circ}_+$
is persistent if every invariant probability measure on the boundary
is a ``repeller''. In a deterministic setting, an equilibrium is a repeller
if it has a positive Lyapunov exponent (or the eigenvalue of the Jacobian).
In a stochastic model, the ergodic invariant measures $\mu\in\M$ play the same role. The $\lambda_i(\mu), i=1,\dots, n$ are the Lyapunov exponents of $\mu$
(it can also be seen that $\lambda_i(\mu)$ gives the long-term growth rate of $X_i(t)$
if $\BX$ is close to the support of $\mu$). As a result,
if $\max_{i=1}^n  \{\lambda_i(\mu)\}>0$,
then the invariant measure $\mu$ is a ``repeller''.
Therefore, Assumption \ref{a.coexn} guarantees the persistence of the population.
\item If an ergodic invariant measure $\mu$ with support on the boundary is an ``attractor'',
it will attract solutions starting nearby.
Intuitively, condition \eqref{ae3.1} forces $X_i(t), i\in I_\mu^c$ to get close to $0$
if the solution starts close to $\R^{\mu,\circ}_+$.
\item In order to characterize extinction we need the additional Assumption \ref{a.extn3} which ensures that apart from those in $\Conv(\M^1)$,
invariant probability measures are ``repellers''.
\end{itemize}
\end{rmk}

\begin{rmk}
The quantity $\lambda_i(\mu)$ can be interpreted as the stochastic growth rate of species $X_i$ when introduced at a low density in the habitat consisting of species $\{X_j, j\in I_\mu\}$. Since $\mu$ is a invariant probability measure, the growth rate of any $X_j, j\in I_\mu$ is $0$. 
\end{rmk}



\begin{example}\label{ex:1}
Let us start by analyzing the one-dimensional equation for the prey

\[
dZ_1(t) = Z_1(t)(a_{10} - a_{11}Z_1(t))\,dt + Z_1(t)\,dE_1(t).
\]

In this case $\M=\{\bdelta^*\}$. One can than easily check that
\[
\lambda_1(\bdelta^*) = a_{10} - \frac{\sigma_{11}}{2} = \tilde a_{10}.
\]
According to \cite[Example 6.2]{HN16} if $\tilde a_{10}>0$ there exists a unique invariant probability measure $\pi^{(1)}$ on $\R_+^{\circ}$ and $Z_1$ converges exponentially fast to $\pi^{(1)}$. If $\tilde a_{10}<0$ then $Z_1$ goes extinct. Next, assume one has the prey and one predator

\begin{equation*}
\begin{split}
dZ_1(t) &= Z_1(t)(a_{10} - a_{11}Z_1(t) - a_{12}Z_2(t))\,dt + Z_1(t)\,dE_1(t)\\
dZ_2(t) &= Z_2(t)(-a_{20} + a_{21}Z_1(t) - a_{22}Z_2(t))\,dt+Z_2(t)\,dE_2(t)\\
\end{split}
\end{equation*}
In view of the analysis from \cite[Example 6.2]{HN16} if $\I_1=\lambda_1(\bdelta^*)=\tilde a_{10}<0$ then $X_1(t), X_2(t)$ converge to 0 almost surely with the exponential rates $\I_1=\tilde a_{10}$ and $\lambda_2(\bdelta^*)=-a_{20}-0.5\sigma_{22}$ respectively.

If $\I_1>0$  then there exists an invariant measure $\mu_1$ on $\R^{\circ}_{1+}:=\{(x_1,0): x_1>0\}$ and
\[
\lambda_1(\mu_1) = \tilde a_{10} - a_{11} \int_{\partial \R_+^{2,\circ}}z d\mu_1 = 0
\]
Now one can compute
\[
\I_2=\lambda_2(\mu_1) = -\tilde a_{20} + a_{21} \int_{\partial \R_+^{2,\circ}}z d\mu_1 = -\tilde a_{20} + a_{21} \frac{\tilde a_{10}}{a_{11}}.
\]

If  $\I_1>0, \I_2<0$ then
$Z_2$ converges to $0$ almost surely with the exponential rate $\lambda_2(\mu_1)$ and the occupation measure of the process $(Z_1,Z_2)$ converges to $\mu_1$.

If $\I_1>0, \I_2>0$ the transition probability of $(Z_1(t), Z_2(t))$ on $\R^\circ_{12+}$ converges to an invariant probability measure in total variation with an exponential rate. The case with two predators is treated in \cite[Example 6.2]{HN16}.

\end{example}

\section{Properties of the invasion rates}\label{s:inv}
We want to say more about the invasion rates $\I_{n+1}$. For this we note by \eqref{e:inv_j+1} that we have to analyze the system \eqref{e:system}. This can be written in matrix form as
\begin{equation}\label{e:inv_mat}
A \bx^{(n)} = \ba
\end{equation}
where $\bx^{(n)} = \left(x_1^{(n)},\dots,x_n^{(n)}\right)^T$, $\ba = (-\tilde a_{10},\tilde a_{20},\tilde a_{30},\dots, \tilde a_{n0})^T$ and
\[
A= \begin{bmatrix}
    -a_{11} & -a_{12} & 0 &\dots  & 0 &0 \\
    a_{21} & -a_{22} & -a_{23}&\dots  & 0 &0 \\
    0 & a_{32} & -a_{33}&\dots  & 0 &0 \\
    \vdots & \vdots & \vdots & \ddots & \vdots & \vdots\\
    0 & 0 & 0 & \dots  &-a_{n-1,n-1}&-a_{n-1,n}\\
     0 & 0 & 0 & \dots  &a_{n,n-1}&-a_{n,n}
\end{bmatrix}
\]
is a tridiagonal $n\times n$ matrix.

It is well-known that the solution can be obtained by a forward sweep that is a special case of Gaussian elimination (see \cite{M01}). To simplify notation we let
\[
(d_1,\dots,d_n)^T :=(-\tilde a_{10},\tilde a_{20},\tilde a_{30},\dots, \tilde a_{n0})^T,
\]
\[
(c_1,\dots,c_{n-1})^T := (-a_{12},-a_{23},\dots,-a_{n-1,n})^T,
\]
\[
(b_1,\dots,b_n)^T:= (-a_{11},\dots,-a_{nn})
\]

and
\[
(f_2,\dots,f_n)^T := (a_{21}, a_{32},\dots,a_{n,n-1})^T.
\]

Define new coefficients $(c_1',\dots,c_{j-1}'), (d_1',\dots,d_j')$ recursively as follows
\begin{equation}\label{e:c'}
c_i'=\left\{
	\begin{array}{ll}
		\frac{c_i}{b_i}, &  i=1\\
		\frac{c_i}{b_i-f_ic_{i-1}'}, & i=2,3,\dots,n-1
	\end{array}
\right.
\end{equation}
and
\begin{equation}\label{e:d'}
d_i'=\left\{
	\begin{array}{ll}
		\frac{d_i}{b_i}, &  i=1\\
		\frac{d_i-f_id_{i-1}'}{b_i-f_ic_{i-1}'}, & i=2,3,\dots,n.
	\end{array}
\right.
\end{equation}
Having defined these coefficients the solution to \eqref{e:inv_mat} can be written as
\begin{equation}\label{e:sol}
\begin{split}
x_n^{(n)} &= d'_n \\
x_i^{(n)} &=d_i'-c_i'x_{i+1}^{(n)}, i=n-1,n-2,\dots,1.
\end{split}
\end{equation}
Since $\I_{n+1}$ only depends on $x_n$, all one needs to do is solve for $d_n'$. In particular,
we can compute directly $\I_n$ for $n\leq 4$ as follows:

\begin{equation}\label{e.i1-4}
\begin{aligned}
\I_1=&\tilde a_{10},\\
\I_2=&-\tilde a_{20}+a_{21}\dfrac{\tilde a_{10}}{a_{11}},\\
\I_3=&-\tilde a_{30}+a_{32}\dfrac{\tilde a_{10}a_{21}-\tilde a_{20}a_{11}}{a_{12}a_{21}+a_{11}a_{22}},\\
\I_4=&-\tilde a_{40}+a_{43}\dfrac{\tilde a_{10}a_{21}a_{32}-a_{11}\tilde a_{20}a_{21}-a_{12}a_{21}\tilde a_{30}-a_{12}a_{21}\tilde a_{30}-a_{11}a_{22}\tilde a_{30}}{a_{12}a_{21}a_{33}+a_{11}a_{22}a_{33}+a_{11}a_{23}a_{32}}.
\end{aligned}
\end{equation}

\begin{rmk}
Since the persistence conditions are $\I_1,\dots, \I_4>0$ we note that they are more likely to be met for the top predator if the growth rate of the prey increases. This agrees with the prediction that the length of a food chain should be an increasing function of the prey growth rate. The fact that the length of a food chain should increase with increasing prey growth rates is a general feature of a multitude of models.
\end{rmk}

\begin{rmk}\label{r:ext}
The invasion rates are functions of the variances $(\sigma_{ii})_{i=1,\dots,n}$. We note that (at least for $j\leq 4$) $\I_j(\sigma_1,\dots,\sigma_j)$ is \textbf{strictly decreasing} in each variable $\sigma_i, 1\leq i\leq 4$. Even though we were not able to give explicit formulas for $\I_n$ one can see from Proposition \ref{p:inv_sol} that for any $1\leq j\leq n$ the quantity
\[
\I_j(\sigma_1,\dots,\sigma_j)
 \]
 is strictly decreasing in the variable $\sigma_u$ for $1\leq u\leq j$ and independent of $\sigma_u$ for $u>j$. As a result environmental stochasticity is seen to increase the risk of extinction.

In the limit of no noise (i.e. $\sigma_{ii}\downarrow 0$ for $1\leq 1\leq 4$) the invasion rates from \eqref{e.i1-4} converge to $\hat \I_i$, that is $\I_i\uparrow \hat \I_i$ as $\sigma_{ii}\downarrow 0$,  where
\begin{equation}\label{e.i1-4_det}
\begin{aligned}
\hat \I_1=& a_{10}>\I_1,\\
\hat \I_2=&- a_{20}+a_{21}\dfrac{ a_{10}}{a_{11}}>\I_2,\\
\hat \I_3=&- a_{30}+a_{32}\dfrac{ a_{10}a_{21}- a_{20}a_{11}}{a_{12}a_{21}+a_{11}a_{22}}>\I_3,\\
\hat \I_4=&- a_{40}+a_{43}\dfrac{ a_{10}a_{21}a_{32}-a_{11} a_{20}a_{21}-a_{12}a_{21} a_{30}-a_{12}a_{21} a_{30}-a_{11}a_{22} a_{30}}{a_{12}a_{21}a_{33}+a_{11}a_{22}a_{33}+a_{11}a_{23}a_{32}}>\I_4.
\end{aligned}
\end{equation}
Since we do not assume $\Gamma$ is positive definite we note that our method also works in the deterministic setting. The expressions for $\hat \I_1,\dots,\hat \I_4$ give, correctly, the deterministic invasion rates.
\end{rmk}

\subsection{Negative invasion rates}\label{s:neg}
For fixed $j\in \{1,\dots,n\}$ write down the system
\begin{equation}\label{ss}
\begin{split}
-a_{11}x_1 - a_{12}x_2 &= -\tilde a_{10}\\
a_{21}x_1 - a_{22}x_2 - a_{23}x_3 &=\tilde a_{20}\\
&\mathrel{\makebox[\widthof{=}]{\vdots}} \\
a_{j-1,j-2}x_{j-2} - a_{j-1,j-1}x_{j-1} -a_{j-1,j}x_{j} &= \tilde a_{j-1,0}\\
a_{j,j-1}x_{j-1} - a_{jj}x_j &= \tilde a_{j0}.
\end{split}
\end{equation}
and suppose it has a strictly positive solution $(x^{(j)}_1,\dots,x^{(j)}_j)$.
The invasion rate of predator $X_{j+1}$ in the habitat of $(X_1,\dots,X_j)$ is given by
\begin{equation}
\I_{j+1} = -\tilde a_{j+1,0} + a_{j+1,j} x^{(j)}_j.
\end{equation}
Now, one can look at the system
\begin{equation}\label{ss2}
\begin{split}
-a_{11}x_1 - a_{12}x_2 &= -\tilde a_{10}\\
a_{21}x_1 - a_{22}x_2 - a_{23}x_3 &=\tilde a_{20}\\
&\mathrel{\makebox[\widthof{=}]{\vdots}} \\
a_{j,j-1}x_{j-1} - a_{j,j}x_{j} -a_{j,j+1}x_{j+1} &= \tilde a_{j,0}\\
a_{j+1,j}x_{j} - a_{j+1,j+1}x_{j+1} &= \tilde a_{j+1,0}.
\end{split}
\end{equation}
and its solution $(x^{(j+1)}_1,\dots,x^{(j+1)}_{j+1})$.
\begin{prop}\label{p:inv_sol}
The following holds
\[
x_{j+1}^{(j+1)}= \frac{\I_{j+1} }{a_{j+1,j+1}+ a_{j+1,j}c_{j}'}
\]
where $c_j'$ is defined in \eqref{e:c'}. In particular, $x_{j+1}^{(j+1)}>0$ if and only if $\I_{j+1}>0$.
\end{prop}
\begin{proof}
Using \eqref{e:d'}
\begin{equation*}\label{e:d'}
d_i'=\left\{
	\begin{array}{ll}
		\frac{d_i}{b_i}, &  i=1\\
		\frac{d_i-f_id_{i-1}'}{b_i-f_ic_{i-1}'}, & i=2,3,\dots,j+1.
	\end{array}
\right.
\end{equation*}
and noting that $d_{j+1}' = x_{j+1}^{(j+1)}$, $d_{j}' = x_{j}^{(j)}$ we get
\[
x_{j+1}^{({j+1})}= \frac{d_{j+1} - f_{j+1} x_{j}^{(j)}}{b_{j+1}-f_{j+1}c_{j}'} =  \frac{-\tilde {a}_{j+1,0} + a_{j+1,j} x_{j}^{(j)}}{a_{j+1,j+1}+ a_{j+1,j}c_{j}'} .
\]
This, together with the expression
\[
\I_{j+1} = -\tilde a_{j+1,0} + a_{j+1,j} x^{(j)}_j.
\]
implies that
\[
x_{j+1}^{(j+1)}= \frac{\I_{j+1}}{a_{j+1,j+1}+ a_{j+1,j}c_{j}'}
\]
Using \eqref{e:c'}  one can easily see that $c'_{j}\geq 0$.
\end{proof}

\begin{prop}\label{p:inv_sol_2}
If there exists $j^*\geq 1$ such that $\I_{j^*+1}<0$ then there exists no solution in $\R_+^{m,\circ}$ for the system \eqref{ss} with $j=m\in\{j^*+1, \dots,n\}$.
\end{prop}
\begin{proof}
By Proposition \ref{p:inv_sol} we note that
\[
x_{j^*+1}^{(j^*+1)}= \frac{\I_{j^*+1} }{a_{j^*+1,j^*+1}+ a_{j^*+1,j}c_{j^*}'} < 0
\]
As a result
\[
\I_{j^*+2} = -\tilde a_{j^*+2,0} + a_{j^*+2,j^*+1} x^{(j^*+1)}_{j^*+1}<0
\]
and by Proposition \ref{p:inv_sol} $x_{j^*+2}^{(j^*+2)}<0$. By repeating this argument we see that $x_m^{(m)}<0$

\end{proof}

\subsection{Case study: equal death, competition and predation rates }
Consider a simplified setting where $a_{ii}=\alpha, i=1,\dots,n$, $a_{i,i-1}=\beta, i=2,\dots,n$, $a_{i,i+1}=\beta, i=1,\dots,n-1$ and $\tilde a_{10}= \delta, \tilde a_{i0}=\gamma, i=2,\dots,n$. In this case we want to solve

\begin{equation}\label{e:inv_mat_2}
A \bx = \ba
\end{equation}
where $\bx = \left(x_1,\dots,x_n\right)^T$, $\ba = (-\delta,\gamma,\gamma,\dots, \gamma)^T$ and
\[
A= \begin{bmatrix}
    -\alpha & -\beta & 0 &\dots  & 0 &0 \\
    \beta & -\alpha & -\beta&\dots  & 0 &0 \\
    0 & \beta & -\alpha&\dots  & 0 &0 \\
    \vdots & \vdots & \vdots & \ddots & \vdots & \vdots\\
    0 & 0 & 0 & \dots  &-\alpha&-\beta\\
     0 & 0 & 0 & \dots  &\beta&-\alpha.
\end{bmatrix}
\]

We use the technique of \cite{D08}[Section 3.1] to find the inverse $A^{-1}$ of $A$. The quadratic equation
\[
\beta - r\alpha - r^2\beta = 0
\]
hast the distinct roots
\[
r_{1,2} = \frac{1}{-2\beta} \left(\alpha \pm\sqrt{\alpha^2+4\beta^2}\right).
\]
One can write the $n$th row of the inverse matrix $A^{-1}$ as
\[
A_{nj}^{-1} = \frac{(r_1^{-j}-r_2^{-j})(r_1^{n+1}r_2^n-r_2^{n+1}r_1^n)}{-\beta(r_1-r_2)(r_1^{n+1}-r_2^{n+1})}, 1\leq j\leq n.
\]
Therefore the solution to \eqref{e:inv_mat_2} satisfies
\[
\begin{split}
x_n &= A_{n1}^{-1}(-\delta) + \sum_{j=2}^n A_{nj}^{-1} \gamma\\
&= -\delta \frac{(r_1^{-1}-r_2^{-1})(r_1^{n+1}r_2^n-r_2^{n+1}r_1^n)}{-\beta(r_1-r_2)(r_1^{n+1}-r_2^{n+1})} + \gamma \frac{(r_1^{n+1}r_2^n-r_2^{n+1}r_1^n)}{-\beta(r_1-r_2)(r_1^{n+1}-r_2^{n+1})}\sum_{j=2}^n (r_1^{-j}-r_2^{-j})\\
&=  -\delta \frac{(r_1^{-1}-r_2^{-1})(r_1^{n+1}r_2^n-r_2^{n+1}r_1^n)}{-\beta(r_1-r_2)(r_1^{n+1}-r_2^{n+1})} + \gamma \frac{(r_1^{n+1}r_2^n-r_2^{n+1}r_1^n)}{-\beta(r_1-r_2)(r_1^{n+1}-r_2^{n+1})} \left(r_1\frac{1-r_1^{n-1}}{1-r_1}-r_2\frac{1-r_2^{n-1}}{1-r_2}\right).
\end{split}
\]
In this case one can write down explicitly the formula for the invasion rate.

\begin{equation*}
\begin{split}
\I_{n+1} =& -\tilde a_{n+1,0} + a_{n+1,n} x^{(n)}_n\\
=& -\gamma + \beta \Bigg[-\delta \frac{(r_1^{-1}-r_2^{-1})(r_1^{n+1}r_2^n-r_2^{n+1}r_1^n)}{-\beta(r_1-r_2)(r_1^{n+1}-r_2^{n+1})} \\
&+ \gamma \frac{(r_1^{n+1}r_2^n-r_2^{n+1}r_1^n)}{-\beta(r_1-r_2)(r_1^{n+1}-r_2^{n+1})} \left(r_1\frac{1-r_1^{n-1}}{1-r_1}-r_2\frac{1-r_2^{n-1}}{1-r_2}\right)\Bigg].
\end{split}
\end{equation*}
\subsection{Trophic cascades in a plant--herbivore--predator system}\label{s:cascade}

Let us explore a food chain with two or three species. We will assume $X_1$ is a plant, $X_2$ is a herbivore eating the plant, and $X_3$ is a predator that preys on the herbivore.

We can compute the expected abundances of different species at stationarity using the linear system \eqref{ss}. We do this to glean information regarding how these expected abundances are changed by intraspecific competition and environmental stochasticity. According to our notation from Section \ref{s:neg} the quantity $x_i^{(j)}$ will denote the abundance of species $i$ at stationarity when there are $j$ species present. In the example with three species we are looking at, $x_1^{(3)}$ will be the abundance at stationarity of the plant when we have the plant, the herbivore, and the predator present. In contrast, $x_1^{(2)}$ is the abundance of the plant at stationarity when we only have the plant and the herbivore present.

Solving the $3\times 3$ or $2\times 2$ system \eqref{ss} directly yields:

$$
x_{3}^{(3)}=\dfrac{\tilde a_{10}a_{32}a_{21} -\tilde a_{30}a_{22}a_{11}-\tilde a_{30}a_{21}a_{12}-\tilde a_{20}a_{32}a_{11}}{a_{33}a_{22}a_{11}+a_{33}a_{21}a_{12}+a_{32}a_{12}a_{11}}
$$
$$
x_{2}^{(3)}=\dfrac{-\tilde a_{20}a_{11}a_{33}+\tilde a_{30}a_{11}a_{12}+\tilde a_{10}a_{21}a_{33}}{a_{33}a_{22}a_{11}+a_{33}a_{21}a_{12}+a_{32}a_{12}a_{11}}
$$
$$x_{1}^{(3)}=\dfrac{\tilde a_{10}a_{33}a_{22}+\tilde a_{10}a_{32}a_{12}+\tilde a_{20}a_{12}a_{33}-\tilde a_{30}a^2_{12}}{a_{33}a_{22}a_{11}+a_{33}a_{21}a_{12}+a_{32}a_{12}a_{11}}$$
$$x_{2}^{(2)}=\dfrac{\tilde a_{10}a_{21}-\tilde a_{20}a_{11}}{a_{22}a_{11}+a_{21}a_{12}}$$
$$x_{1}^{(2)}=\dfrac{\tilde a_{10}a_{22}+\tilde a_{20}a_{12}}{a_{22}a_{11}+a_{21}a_{12}}$$

$$ x_1^{(1)}= \frac{\tilde a_{10}}{a_{11}}.$$

Let us explore the effect the introduction of a predator $X_3$ has on the expected density of the plant and the herbivore at stationarity. Using the formulas above one can show that, as long as the predator $X_3$ persists, i.e. $x_3^{(3)}>0$, we will always have $x_1^{(3)}-x_1^{(2)}>0$. If the predator $X_3$ goes extinct, i.e. $x_3^{(3)}=0$, then $x_1^{(3)}-x_1^{(2)}=0$. Similarly, one can show that if $x_3^{(3)}>0$ then $x_2^{(3)}-x_2^{(2)}<0$ and if $x_3^{(3)}=0$ then $x_2^{(3)}-x_2^{(2)}=0$. One can also note that the abundance of the plant species is decreasing as we increase the death rate of the predator.

In order to get more information, we graph $x_1^{(3)}-x_1^{(2)}$ and $x_2^{(3)}-x_2^{(2)}$ as functions of the predation rate of the predator on the herbivore,$a_{32}$, and the intracompetition rate of the predator, $a_{33}$. See Figures \ref{fig:4} and \ref{fig:6}. Similarly, in Figure \ref{fig:5} (respectively Figure \ref{fig:7}) we graph $x_1^{(3)}-x_1^{(2)}$ (respectively $x_1^{(3)}-x_1^{(2)}$) as a function of predation rate $a_{32}$, the death rate of the predator $\tilde a_{33}$ or the intracompetition rate of the predator $a_{33}$. For Figures \ref{fig:1}, \ref{fig:4}, \ref{fig:5}, \ref{fig:6} and \ref{fig:7} we have set $\tilde a_{10}=4$ and all the other coefficients (other than the ones being varied) equal to $1$. For Figure \ref{fig:3} we have set all constant coefficients equal to $1$.

We note that the introduction of the predator is always beneficial to the plant and detrimental to the herbivore. The predator will decrease the population size of the herbivore, which will lead to an increase in the plant population size. The density of the plant is seen to increase as we increase the predation rate $a_{32}$ of the predator on the herbivore, and as we decrease the intraspecific competition rate, $a_{33}$, among predators. Plant density will also increase if the stochastic death rate $\tilde a_{30}$ of the predator decreases. The areas of the graphs from Figures \ref{fig:4} and \ref{fig:5} where $x_1^{(3)}-x_1^{(2)}=0$ (or from Figures \ref{fig:6} and \ref{fig:7} where $x_2^{(3)}-x_2^{(2)}=0$) are those where the predator $X_3$ goes extinct. It turns out that anything that is detrimental to the predator (higher intracompetition rate or higher death rate), is also detrimental to the plant. Similarly, factors that are helping the predator survive (higher predation rate $a_{32}$) increase the density of the plant. If one looks at the herbivore $X_2$ then its abundance at stationarity will always suffer by the introduction of the predator.

\begin{figure}
\begin{center}
\includegraphics[scale=0.6 ]{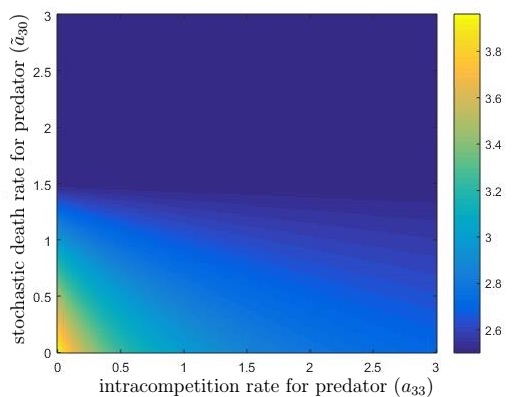}
\caption{Expected density of plant $X_1$ at stationarity as a function of the intraspecific competition $a_{33}$ and stochastic death rate $\tilde a_{30}$ of the predator.}
\label{fig:1}
\end{center}
\end{figure}

\begin{figure}
\begin{center}
\includegraphics[scale=0.6]{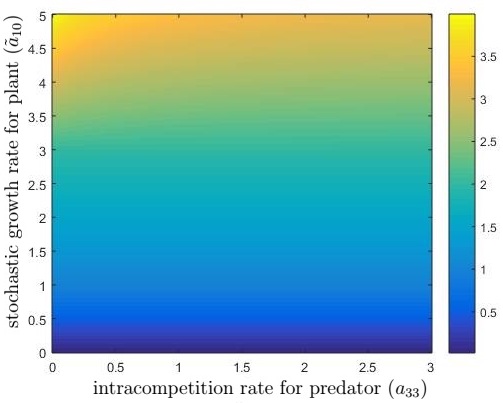}
\caption{Expected density of plant $X_1$ at stationarity as a function of its stochastic growth rate $\tilde a_{10}$ and the intraspecific competition $a_{33}$ of the predator.}
\label{fig:3}

\end{center}
\end{figure}

\begin{figure}
\begin{center}
\includegraphics[scale=0.6]{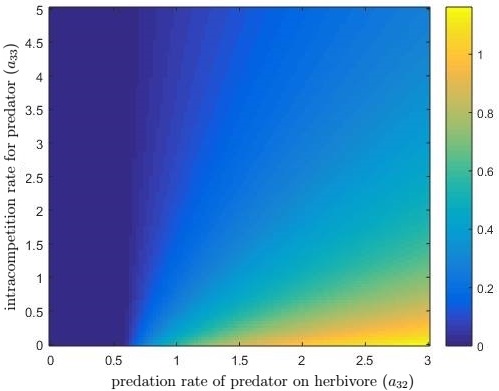}
\caption{Difference between expected densities of plant species $x_1^{(3)}-x_1^{(2)}$ with or without a predator.}
\label{fig:4}

\end{center}
\end{figure}

\begin{figure}
\begin{center}
\includegraphics[scale=0.6]{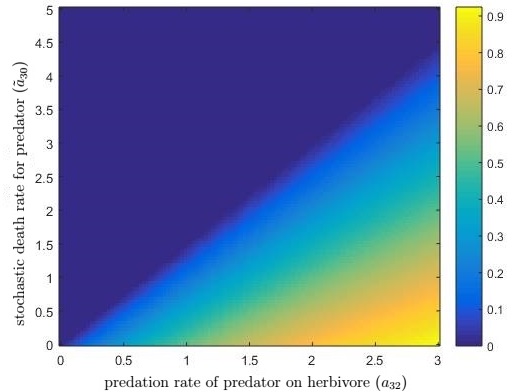}
\caption{Difference between expected densities of plant species $x_1^{(3)}-x_1^{(2)}$ with or without a predator.}
\label{fig:5}

\end{center}
\end{figure}

\begin{figure}
\begin{center}
\includegraphics[scale=0.6]{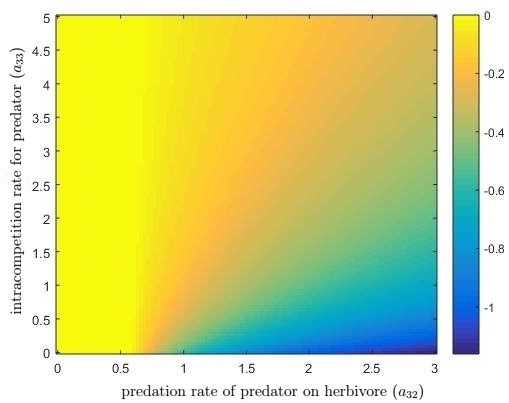}
\caption{Difference between expected densities of herbivore species $x_2^{(3)}-x_2^{(2)}$ with or without a predator.}
\label{fig:6}

\end{center}
\end{figure}

\begin{figure}
\begin{center}
\includegraphics[scale=0.6]{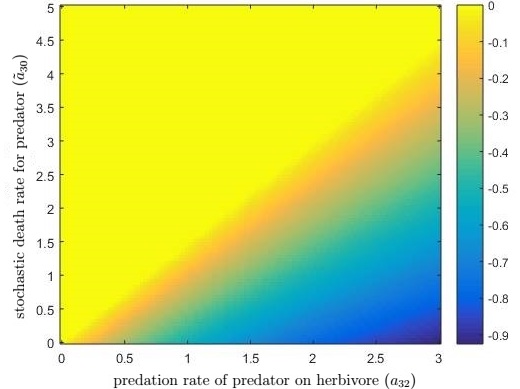}
\caption{Difference between expected densities of herbivore species $x_2^{(3)}-x_2^{(2)}$ with or without a predator.}
\label{fig:7}

\end{center}
\end{figure}

\section{Discussion}\label{s:disc}
Even though environmental stochasticity is often said to be a key factor in the study of the persistence of species, its effect on persistence has not been investigated until recently. \cite{BHS08} showed that if one adds a small diffusion term to a persistent deterministic system then the corresponding differential equation has a positive stationary distribution concentrated on the positive global attractor of the deterministic system. For many systems the random perturbations might not be small. For populations living in a compact state space \cite{SBA11} give sufficient conditions for persistence that extend the results from deterministic systems to randomly forced nonlinear systems. This has been further extended by \cite{HN16} to non-compact states spaces. \cite{HN16} are able to give, under some mild assumptions, sufficient and necessary conditions for persistence and extinction for stochastic Kolmogorov systems of the form
\[
d\BX(t) = \BX(t)\mathbf{f}(\BX(t))\,dt +  \BX(t)\mathbf{g}(\BX(t))\,d\mathbf{E}(t).
\]

Most results which give sharp, tractable conditions for persistence and extinction of populations in stochastic environments usually treat models with only two species (see \cite{EHS15, R03}). However, it has been shown that it makes more sense to look at food chains having more than two species (\cite{HP91, KH94, P88}). We note that although our model formulation ignores many important ecological features, it still yields some interesting conclusions. These mathematical conclusions should be further studied as hypthesis requiring confirmation. Our results are a first step towards an analysis of stochastic food chains and food webs with an arbitrary number of species.

In the current paper we use the newly developed methods from \cite{HN16} to analyse the persistence and extinction of species that are part of a stochastic Lotka-Volterra food chain. We assume that species can only interact with those other species which are adjacent to them in the food chain and that there is strictly positive intraspecies competition for all the species. Our main interest was to lift the results from the deterministic setting to the stochastic one and to see whether stochasticity inhibits or enhances coexistence. By studying the invasion rates of the predators $(\I_2,\dots,\I_n)$ we show that one can determine which species persist and which go extinct exponentially fast. Furthermore, we provide in Section \ref{s:inv} an algorithm for computing the invasion rates. In this way, based on the interaction coefficients of the system, one can find sufficient and (almost) necessary conditions for persistence/extinction.
We show that the introduction of a new top predator into the ecosystem makes extinction more likely. This agrees with the deterministic case studied in \cite{GH79}. Furthermore, we also note that in our setting stochasticity makes extinction more likely.  However, since the invasion rates depend continuously on the covariance matrix $\Sigma$ of the environmental noise one can see that if the random perturbations are small and the associated deterministic system is persistent then the stochastic system is also persistent. Actually, our results (see Remark \ref{r:ext}) show that stochasticity acts in a bottom-up way: the variance $\sigma_{11}$ affects species $\{1,\dots,n\}$ and the variance $\sigma_{jj}, j=2,\dots,n$ affects the species $\{j,j+1,\dots,n\}$. As such, in our model the environmental stochasticity of a trophic level only affects the persistence and extinction of species at higher trophic levels.
We note that environmental stochasticity does not always make extinction more likely. For example in \cite{BL16} the authors show that in certain cases, the extinction of species in a deterministic setting can be reversed into coexistence by adding randomness to the system. As such,we think that the rigorous study of the stochastic system we propose, did provide insightful information.

It is noteworthy that we do not only get robust results for extinction or persistence -- we also get that the convergence to the stationary distribution in the case of persistence is exponentially fast and an exact expression for the convergence rate to $0$ in the case of extinction. These rates are very helpful when one wants to run numerical methods and simulations.

We have fully analysed what happens in chains of length $n\leq 4$. Intraspecific competition is shown to change both the conditions for persistence and the strength of trophic cascades.

 Humans have always tried to exterminate predators: slayers of predators were seen as heroes in most mythologies; culls have been used to control seals and sea lions in order to manage fisheries; predator control agents are often hired to kill predators (wolves, coyotes, etc). We have effectively decimated and in some cases even driven to extinction entire species of predators. The effects of these exterminations are now becoming more and more clear. There are a plethora of reasons why predators are important in food webs. Predators are usually at the top of the food chain and thus can regulate the trophic levels below them. Removing predators often destabilizes the food chain, and sets off reactions that can cascade down to the lowest trophic level. In Section \ref{s:cascade} we looked in depth at a plant--herbivore--predator food chain. What our computations and figures show is the following: Aything that helps the predator (decreased death rate, higher predation rate) will be detrimental to the herbivore and favorable to the plant.

Our model also leads to the observation that food chain length should increase when we increase the stochastic growth rate $\tilde a_{10}$ of the plant at the first trophic level.

The only cases that cannot be treated are those for which one of the invasion rates is zero, that is $\I_k=0$ for some $k\in\{2,\dots,n\}$. This is where our methods break down. As mentioned in \cite{GH79} even in the deterministic setting, when $\kappa(n)=0$ (which would imply one of the invasion rates is $0$) the problem becomes more complicated: one can find solutions with positive initial conditions which persist when $n=3$ while when $n=4$ there are solutions which are not persistent. In the stochastic case, when $n=1$ the prey is described by the SDE
\[
d X_1(t) =  X_1(t)(a_{10} - a_{11} X_1(t) )\,dt +  X_1(t)\,dE_1(t).
\]
If $\I_1=\tilde a_{10}=0$ then one can show that $X_1$ is null recurrent and
\begin{equation}\label{e:X1}
\lim_{t\to \infty} \frac{1}{t}\int_0^t X_1(s)\,ds =0.
\end{equation}
As a result the prey $X_1$ is not strongly stochastically persistent (there is no invariant probability measure on $(0,\infty)$) but it also does not go extinct. It only goes extinct in the weak sense given by \eqref{e:X1}. We expect similar phenomena to occur in higher dimensions if one of the invasion rates is zero.
 One possible approach would be to try and adapt the methods used by \cite{B91} where the author studies SDE where the extinction set is $\{0\}$. \cite{B91} is able to show that if the leading Lyapunov exponent is zero then the process is null-recurrent. In the setting of \cite{B91} one only has to study the dirac measure at $0$, something which simplifies the problem significantly.

Our results generalize the results from the deterministic setting of \cite{GH79} to their natural stochastic analogues. We are able to find an algebraically tractable criterion (just like in the deterministic setting) for persistence and extinction.

The invasion rates are shown to be closely related to the first moments of the invariant measures living on the boundary $\partial\R_+^n$ of the system. This result is the analogue of looking for the different equilibrium points of the deterministic system \eqref{e:det} and then studying the stability of these points.

The main simplification of our model is the fact that the dynamics of each trophic level is governed by the adjoining trophic levels which immediately precede or succeed it. This fact makes it possible to explicitly describe the structure of the ergodic invariant probability measures of the system living on the boundary $\partial\R_+^n$ (Lemma \ref{l:inv}). The key property of an invariant probability measure $\mu$ living on $\partial\R_+^n$ is that if predator $X_j$ is not present then all predators that are above $j$ (that is, $X_i$ with $i>j$) are also not present. This fact is biologically clear because if species $X_j$ does not exist then $X_{j+1}$ must go extinct since it does not have a food source.

For more complex interactions between predators and their prey (i.e. a food web instead of a food chain), even when $n=3$, the possible outcomes become much more complicated. We refer the reader to \cite{HN16} for a detailed discussion of the case when one has one prey and two predators and the apex predator eats both the intermediate predator and the prey.

In ecology there has been an increased interest in the \textit{spatial synchrony} that appears in population dynamics. This refers to the changes in the time-dependent characteristics (i.e. abundances etc) of structured populations. One of the mechanisms which creates synchrony is the dependence of the population dynamics on a synchronous random environmental factor such as temperature or rainfall. The synchronizing effect of environmental stochasticity, or the so-called \textit{Moran effect}, has been observed in multiple population models. Usually this effect is the result of random but correlated weather effects acting on populations. For many biotic and abiotic factors, like population density, temperature or growth rate, values at close locations are usually similar. We refer the reader interested in an in-depth analysis of spatial synchrony to \cite{K00, LKB04}.
Most stochastic differential equations models appearing in the population dynamics literature treat only the case when the noise is non-degenerate (although see \cite{R03, DNDY16}). Although this approach significantly simplifies the technical proofs, from a biological point of view it is not clear that the noise should not be degenerate. For example, if one models a system with multiple populations then all populations can be influenced by the same factors (a disease, changes in temperature and sunlight etc). Environmental factors can intrinsically create spatial correlations and as such it makes sense to study how these degenerate systems compare to the non-degenerate ones. In our setting the noise of the different species could be strongly correlated. Actually, in some cases it could be more realistic to have the same one-dimensional Brownian motion $(B_t)_{t\geq 0}$ driving the dynamics of all the interacting species. Therefore, we chose to present a full analysis of the degenerate setting.

{\bf Acknowledgments.}  We thank an anonymous referee for comments which helped improve this manuscript and  Sebastian Schreiber for helpful discussions and suggestions.

\bibliographystyle{amsalpha}
\bibliography{LV}

\appendix
\section{Proofs}\label{s:proofs}

The following result tells us that there is no ergodic invariant probability measure $\mu$ that has a gap in the chain of predators.

\begin{lm}\label{l:inv}
Suppose $\mu\in\M$ such that $I_\mu= \{n_1,\dots,n_k\}$. Then $I_\mu$ must be of the form $\{1,2,\dots,l\}$ for some $l\geq 1$.
\end{lm}
\begin{proof}
We argue by contradiction. First, suppose that $n_1>1$. By \eqref{e:lambda_0}
\begin{equation*}
\begin{split}
\lambda_{n_1}(\mu)= 0 &=-\tilde a_{n_1,0} + a_{n_1,n_1-1}\int_{ \R_+^n} x_{n_1-1}d\mu - a_{n_1,n_1}\int_{ \R_+^n} x_{n_1}d\mu\\
&= -\tilde a_{n_1,0} - a_{n_1,n_1}\int_{ \R_+^n} x_{n_1}d\mu \\
&<0
\end{split}
\end{equation*}
which is a contradiction.

Alternatively, suppose that there exists $\mu\in\M$ such that $I_\mu= \{1,\dots, u^*, v^*,\dots, n_k\}$ with $1\leq u^*<v^*-1\leq n_k\leq n$. As a result one can see that $v^*-1\notin I^\mu$. Then by \eqref{e:lambda_0}
\begin{equation*}
\begin{split}
\lambda_{v^*}(\mu)= 0 &=-\tilde a_{n_1,0} + a_{v^*,v^*-1}\int_{ \R_+^n} x_{v^*-1}d\mu - a_{v^*,v^*}\int_{ \R_+^n} x_{v^*}d\mu - a_{v^*,v^*+1}\int_{ \R_+^n} x_{v^*+1}d\mu\\
&= -\tilde a_{v^*,0} - a_{v^*,v^*}\int_{ \R_+^n} x_{v^*}d\mu - a_{v^*,v^*+1}\int_{ \R_+^n} x_{v^*+1}d\mu\ \\
&<0
\end{split}
\end{equation*}
which is a contradiction.
\end{proof}
For $i=1,\dots,n$, denote by $\M_i$ the set of all invariant probability measures $\mu$ of $\BX$
satisfying $\mu\left(\R^{(i),\circ}_+\right)=1$.
For $i=0$, define $\M_0=\{\bdelta^*\}$.
By Lemma \ref{l:inv},
we have $\Conv(\M)=\Conv(\cup_{i=0}^{n-1}\M_i)$
and
$\Conv(\cup_{i=0}^{n}\M_i)$ is the set of all invariant probability measures of $\BX$ on $\R^n_+$.
\begin{lm}\label{lm3.2}
We have the following claims.
\begin{itemize}
\item
If $\I_k\leq 0$ then $\I_{k+1}<0$.
\item If $\I_n\leq 0$,
there $\BX$ has no invariant probability measure on $\R^{n,\circ}_+$.
\end{itemize}
\end{lm}
\begin{proof}
If
$\I_{k+1}=
-\tilde a_{k+1,0} + a_{k+1,j} x^{(k)}_k\geq 0$,
then $x^{(k)}_k>0$.
We will show in Section 4 that $x^{(k)}_k$
 has the same sign as $\I_k$.
Thus, if $\I_{k+1}\geq 0$ then $\I_k>0$, which proves the first claim.

If $\BX$ has an invariant probability measure $\mu$ on $\R^{n,\circ}_+$,
then we must have
$\int_{\R^n_+}x_n\mu(d\bx)=x^{(n)}_n$.
As a result $x^{(n)}_n>0$,
which leads to $\I_n>0$ since they have the same sign.
The second claim is therefore proved.
\end{proof}
\begin{lm}\label{lm3.3}
We have the following claims.
\begin{enumerate}
\item For any initial condition $\BX(0)=\bx\in\R^n_+$,
the family $\left\{\wtd \Pi_t(\cdot), t\geq 1\right\}$ is tight in $\R^n_+$,
and its weak$^*$-limit set, denoted by $\U=\U(\omega)$
is a family of invariant probability measures of $\BX$ with probability 1.
\item Suppose that there is a sequence $(T_k)_{k\in\N}$ such that $\lim_{k\to\infty}T_k=\infty$ and
$\left(\wtd \Pi_{T_k}(\cdot)\right)_{k\in\N}$ converges weakly to an invariant probability measure $\pi$ of $\BX$
when $k\to\infty$ .
Then for this sample path, we have
$\int_{\R^n_+}h(\bx)\wtd\Pi_{T_k}(d\bx)\to \int_{\R^n_+}h(\bx)\pi(d\bx)$
for any continuous function $h:\R^n_+\to\R$ satisfying
$|h(\bx)|<K_h(1+\|\bx\|)\,,\,\bx\in \R^n_+$,
with $K_h$ a positive constant and $\delta\in[0,\delta_1)$.
\item
For any $\bx\in\R^{n,\circ}_+$
\begin{equation}\label{e.rem}
 \PP_\bx\left\{\lim_{t\to\infty}\left(\dfrac{\ln X_i(t)}t-\lambda_i\left(\tilde\Pi_t\right)\right)=0,\, i=1,\dots,n\right\} =1
\end{equation}
and
\begin{equation}
\label{e4-thm1.1}
 \PP_\bx\left\{\limsup_{t\to\infty} \dfrac{\ln X_i(t)}t\leq 0, i=1,\dots,\right\}n=1.
\end{equation}

\end{enumerate}
\end{lm}
\begin{proof}
Let $\tilde c_1=1, \tilde c_i:=\prod_{j=2}^i\dfrac{a_{k-1,k}}{2a_{k,k-1}}=c_{i-1}\dfrac{a_{i-1,i-1}}{2a_{i,i-1}}, i\geq 2.$
Put $$\tilde\gamma=\min_{i=1,\dots,n}\left\{{c_i}\frac{a_{ii}}2\right\}$$
we can easily verify that
$$\sum_{i=1}^n\tilde c_i f_i(\bx)\leq \tilde C -\tilde\gamma\sum_{i=1}^n x_i\,\text{ for some positive constant }\, \tilde C.$$
Thus,
when $\|x\|$ is sufficiently large,
$|\sum_{i=1}^n\tilde c_i f_i(\bx)|\geq \tilde\gamma\sum_{i=1}^n x_i$,
which implies
$$
\liminf_{\|x\|\to\infty} \dfrac{\sum_{i=1}^n\tilde c_i |f_i(\bx)|}{\sum_{i=1}^n x_i}
\geq \liminf_{\|x\|\to\infty} \dfrac{\left|\sum_{i=1}^n\tilde c_i f_i(\bx)\right|}{\sum_{i=1}^n x_i}\geq \tilde\gamma.
$$
As a result,
$$\liminf_{\|x\|\to\infty} \dfrac{\|\bx\|^\delta}{\sum_{i=1}^n |f_i(\bx)|}=0\,\text{ for any }\,\delta\in(0,1).
$$
In other words,
Assumption 1.4 of \cite{HN16} is satisfied by our model.
Thus, the first and second claims of this lemma
follow from \cite[Lemma 4.6, Lemma 4.7]{HN16}.
By It\^o's formula
and the definition of $\tilde\Pi_t$, we have
$$
\left(\dfrac{\ln X_i(t)}t-\lambda_i\left(\tilde\Pi_t\right)\right)=
\dfrac{\ln X_i(0)}t+\frac{E_i(t)}{ t}.
$$
By the strong law of large numbers for martingales,
$$\lim_{t\to\infty}\dfrac{\ln X_i(0)}t+\frac{E_i(t)}{ t}=0 \,\text{ a.s.}$$
which leads to \eqref{e.rem}.

\eqref{e4-thm1.1}
can be derived by using equation (4.22) of \cite{HN16} or by mimicking the proof of \cite[Theorem 2.4]{DS06}.
\end{proof}
\begin{proof} [Proof of Theorem \ref{t:main} (i)]
Since $\I_n>0$,
it follows from Lemma \ref{lm3.2} that $\I_k>0$
for any $k=1,\dots,n$.
By Lemma \ref{l:inv},
for any $\mu\in\Conv(\M)=\Conv(\cup_{i=0}^{n-1}\M_i)$,
we can decompose $\mu=\rho_1\mu_{i_1}+\dots+\rho_k\mu_{i_k}$
where $0\leq i_1<\dots<i_k\leq n-1$
and
$\mu_{i_j}\in\M_{i_j}$, $\rho_j> 0$ for $j=1,\dots,k$ and $\sum\rho_j=1$.
Since
$i_1< i_j$ for $j=2,\dots,k$,
we deduce from \eqref{e:lambda_0}
that $\lambda_{i_1+1}(\mu_{i_j})=0$ for $j=2,\dots,k$.
On the other hand,
\eqref{e:inv_j+1} and \eqref{e.meanmu} imply
$$\lambda_{i_1+1}(\mu_{i_1})= -\tilde a_{i_1+1,0} + a_{i_1+1,i_1} x^{(i_1)}_{i_1}=\I_{i_1+1}>0.$$
As a result,
$$\lambda_{i_1+1}(\mu)=\rho_1\lambda_{i_1+1}(\mu_{i_1})>0.$$
Thus,
\begin{equation}\label{e6-thm1.1}
\max_{i=1,\dots,n}\lambda_i(\mu)>0,\text{ for any }\,\mu\in\Conv(\M).
\end{equation}
In other words,
Assumption \ref{a.coexn} is satisfied.
By Theorem 3.1 of \cite{HN16},
there exist positive $p_1,\dots, p_n, T$ and constants $\theta, \kappa\in (0,1)$
such that
\begin{equation}\label{e5-thm1.1}
\E_{\bx} V^\theta(X( T))\leq \kappa V^\theta(x)+ K
\end{equation}
where
$$
V(\bx):=\dfrac{1+\bc^\top\bx}{\Pi_{i=1}^nx_i^{p_i}}\,\text{ for }\, \bx\in\R^{n,\circ}_+,\,\text{ with } \bc\,\text{ defined in }\, \eqref{e.c},\,\text{ and }\, \sum_{i=1}^np_i<1.
$$
Equation \eqref{e5-thm1.1} and the Markov property of $\BX$ lead to
$$\E_{\bx} V^\theta(X(m T))\leq \kappa^m V^\theta(x)+K\sum_{j=1}^{m-1}\kappa^j.$$
Thus,
\begin{equation}\label{e1-thm1.1}
\limsup_{m\to\infty}\E_{\bx} V^\theta(\BX(m T))\leq \dfrac{ K}{1-\kappa},\,\bx\in\R^{n,\circ}_+.
\end{equation}
By \cite[Lemma 2.1]{HN16},
there exists $\hat K>0$ such that
$$\E_{\bx} V^\theta(\BX(t))\leq \exp(\hat Kt)V^\theta(\bx),\,\bx\in\R^{n,\circ}_+,$$
which together with the Markov property implies
\begin{equation}
\label{e2-thm1.1}
\E_{\bx} V^\theta(\BX(t))\leq \exp(\hat K T)\E_\bx V^\theta(\BX(m T))
\text{ for } t\in[m T,(m+1) T].
\end{equation}
In view of \eqref{e1-thm1.1}
and \eqref{e2-thm1.1},
we have
$$
\limsup_{t\to\infty}\E_{\bx} V^\theta(\BX(t))\leq \exp(\hat K T)\dfrac{ K}{1-\kappa}.$$

For any fixed $\eps>0$,
define $K:=\left\{\bx\in\R^{n,\circ}_+: V^\theta(\bx)\leq\dfrac1\eps\exp(\hat K T)\dfrac{ K}{1-\kappa}\right\}$
then $K$ is a compact subset of $\R^{n,\circ}_+$.
The definition of $K$ together with the last inequality yield
\begin{equation}
\label{e3-thm1.1}
\limsup_{t\to\infty}\PP_{\bx} \{\BX(t)\notin K\}\leq
\left(\eps\exp(-\hat K T)\dfrac{1-\kappa}{ K}\right)\limsup_{t\to\infty}\E_{\bx} V^\theta(\BX(t))\leq \eps.
\end{equation}
The stochastic persistence in probability is therefore proved.

To prove
\eqref{e0a-thm1.1},
we need to show that
for any initial value $\bx\in\R^{n,\circ}_+$,
the weak-limits points of
$\tilde\Pi_{t}$ are a subset of $\M_n$
with probability 1.

Suppose the claim is false.
Then, by part (i) of Lemma \ref{lm3.3}, we can find $\bx\in\R^{n,\circ}_+$ and $\tilde\Omega_\bx\subset \Omega$ with $\PP_\bx(\tilde\Omega_\bx)>0$ and
such that for $\omega\in \tilde\Omega_\bx$,
there exists $t_k=t_k(\omega)$
satisfying that $\lim_{k\to\infty}t_k=\infty$
and $\tilde\Pi_{t_k}(\omega)$ converges weakly to $\mu(\omega)=\rho_1\mu_1+\rho_2\mu_2$
where $\mu_1\in\Conv(\M)$ and $\mu_2\in\M_n$ and $\rho_1>0$.
By Lemma \ref{lm3.2}, $\lambda_n(\mu_1)>0$.
In view of \eqref{e:lambda_0}, $\lambda_n(\mu_2)=0$.
Thus, for almost all $\omega\in\tilde\Omega_\bx$, we have from part (ii) of Lemma \ref{lm3.3} that
$$\lim_{k\to\infty}\dfrac{\ln X_n(t_k)}{t_k}
=\lim_{k\to\infty}\lambda_n\left(\tilde\Pi_{t_k}\right)=\lambda_n(\mu)>0,$$
which contradicts \eqref{e4-thm1.1}.
Thus, with probability 1,
the weak-limit points of $\tilde\Pi_{t}$ as $t\to\infty$
must be contained in $\M_n$.
Then, \eqref{e0a-thm1.1} follows from \eqref{e.meanmu}.

When $\Sigma$ is positive definite,
it follows from \cite[Theorem 3.1]{HN16} that
the food chain $\BX$
is strongly stochastically persistent and its transition probability converges to its unique invariant probability measure $\pi^{(n)}$ on $\R_+^{n,\circ}$ exponentially fast in total variation.

\end{proof}

\begin{proof} [Proof of Theorem \ref{t:main} (ii)]

We suppose there exists $j^*<n$ such that $\I_{j^*}>0$ and $\I_{j^*+1}<0$. By Lemma \ref{lm3.2} part (ii),
there are no invariant probability measures on $\R^{(j),\circ}_+$
for $j=j^*+1,\dots,n$.
Using Lemma \ref{l:inv},
we see that the set of invariant probability measures on $\R^n_+$
of $\BX$ is $\Conv(\cup_{i=0}^{j^*}\M_i)$.

Note that $\lambda_{j^*+1}(\mu)=-\tilde a_{j^*+1}<0$ if $\mu\in \M_i$ for $i<j^*$
and $\lambda_{j^*+1}(\mu)=\I_{j^*+1}<0$ if $\mu\in\M_{j^*}$.
As a result,
$\lambda_{j^*+1}(\mu)<0$ for any $\mu\in\Conv(\cup_{i=0}^{j^*}\M_i)$.
Similarly,
$\lambda_{j}(\mu)<0$ for any $j>j^*+1$ and $\mu\in\Conv(\cup_{i=0}^{j^*}\M_i)$.
By \eqref{e.rem}
we have that
$$\lim_{t\to\infty}X_j(t)=0, j=j^*+1,\dots,n \,\,\PP_\bx-\text{a.s.}$$

Since
\begin{equation}
\label{e7-thm1.1}
\int_{\R^n_+}x_i'\mu(d\bx')=
\begin{cases}
x^{(j^*)}_i\,&\text{ if } i=1,\dots,j^*,\\
0\,&\text{ if } i=j^*+1,\dots, n.
\end{cases}
\,\text{ for }\, \mu\in\M_{j^*},
\end{equation}
we have
\begin{equation}
\label{e8-thm1.1}
\lambda_{i}(\mu)=
\begin{cases}
\I_{j^*+1}\,&\text{ if } i=j^*+1\\
-\tilde a_{i0}
\,&\text{ if } i>j^*+1.
\end{cases}
\,\text{ for }\, \mu\in\M_{j^*}.
\end{equation}
Using \eqref{e4-thm1.1}
and a contradiction argument similar to that in the proof of part (i),
we can show that
with probability 1,
the weak-limit points of $\tilde\Pi_{t}$ as $t\to\infty$
must be contained in $\M_{j^*}$.
Thus, for $\bx\in\R^{n,\circ}_+$,
we have from \eqref{e7-thm1.1}, \eqref{e8-thm1.1},
and
Lemma \ref{lm3.2} that $$
\lim_{t\to\infty}\dfrac1t\int_0^t X_i(s)ds=
\begin{cases}
x^{(j^*)}_i\,&\text{ if } i=1,\dots,j^*,\\
0\,&\text{ if } i=j^*+1,\dots, n
\end{cases}\,\,\PP_\bx-\text{a.s}.
$$
and
\begin{equation*}
\lim_{t\to\infty}\dfrac{\ln X_i(t)}t=
\begin{cases}
\I_{j^*+1}\,&\text{ if } i=j^*+1\\
-\tilde a_{i0}
\,&\text{ if } i>j^*+1.
\end{cases}
\,\,\PP_\bx-\text{a.s}.
\end{equation*}

To prove the persistence in probability of $(X_1,\dots,X_{j^*})$,
we define

$$\R^{(j^*),\diamond}=\Big\{\bx=(x_1,\dots,x_n)\in\R^n_+: x_j>0 \,\text{ for }\, j=1,\dots,j^*\Big\},\,\,\text{ and }\, \partial\R^{(j^*),\diamond}=\R^n_+\setminus \R^{(j^*),\diamond}.$$
We have proved that $\Conv\left(\bigcup_{j=0}^{j^*}\M_j\right)$
is the set of invariant probability measures of $\BX$ on $\R^n_+$.
Note that $\Conv\left(\bigcup_{j=0}^{j^*-1}\M_j\right)$
is the set of invariant probability measures of $\BX$ on $\partial\R^{(j^*),\diamond}$.
Since $\I_{j^*}>0$,
applying \eqref{e6-thm1.1} with $n$ replaced by $j^*$ we obtain
\begin{equation}\label{e9-thm1.1}
\max_{i=1,\dots,j^*}\lambda_i(\mu)>0,\text{ for any }\,\mu\in\Conv\left(\cup_{j=0}^{j^*-1}\M_j\right).
\end{equation}
Using this condition,
we can imitate the proofs in \cite[Section 3]{HN16} to construct a Lyapunov function
$U(\bx):\R^{(j^*),\diamond}_+\mapsto\R_+$ of the form
$$U(\bx)=\dfrac{1+\bc^\top\bx}{\Pi_{i=1}^{j^*}x_i^{\tilde p_i}}, \tilde p_i>0, i=1,\dots,j^*$$
satisfying
\begin{equation}\label{e10-thm1.1}
\E_{\bx} U^{\tilde\theta}(X( T))\leq \tilde \kappa U^{\tilde\theta}(x)+\tilde K, \text{ for }\, \bx\in\R^{(j^*),\diamond}_+
\end{equation}
and
\begin{equation}\label{e11-thm1.1}
\E_{\bx} U^{\tilde\theta}(X( t))\leq \exp(\bar K t) U^{\tilde\theta}(x)\,\text{ for }\, \bx\in\R^{(j^*),\diamond}_+ ,
\end{equation}
where
$\tilde p_i>0$ for $i=1,\dots,j^*$,  $\sum_{i=1}^{j^*}\tilde p_i<1$,
$\tilde\theta, \tilde\kappa$ are some constants in $(0,1)$,
and $\tilde T, \tilde K, \bar K$ are positive constants.
Using \eqref{e10-thm1.1} and \eqref{e11-thm1.1},
we can obtain the persistence in probability of $(X_1,\dots,X_{j^*})$
in the same manner as \eqref{e3-thm1.1}.
The proof is complete.
\end{proof}

\begin{proof} [Proof of Theorem \ref{t:main} (iii)]
Let $f:\R^n_+\mapsto\R$ be a continuous function
and $\sup_{\bx\in\R^n_+}|f(\bx)|\leq 1$.
Fix $\bx_0\in\R^{n,\circ}_+$.
We have to show that
\begin{equation}\label{e19-thm1.1}
\lim_{t\to\infty}\left|\int_{\R^n_+} f(\bx')\pi_{j^*}(d\bx')-\int_{\R^n_+}f(\bx')P(t, \bx_0, d\bx')\right|=0.
\end{equation}
In part (ii),
we have proved that
$(X_1,\dots, X_{j^*})$ is persistent in probability.
Thus, for any $\eps>0$, there exist $T_1>0$ and $H>1$ such that
\begin{equation}\label{e13-thm1.1}
\PP_{\bx_0}\left\{H^{-1}\leq X_j(t)\leq H, j=1,\dots, j^*\right\}>1-\eps\,\text{ for any }\,t\geq T_1.\end{equation}
For $\delta\geq0$ define
$$K_\delta=\{\bx=(x_1,\dots,x_n)\in\R^n_+: H^{-1}\leq x_j\leq H,\text{ for } j=1,\dots, j^*, x_j\leq \delta,\text{ for } j=j^*+1,\dots,n\}.$$
Let $\bar f=\int_{\R^n_+} f(\bx')\pi_{j^*}(d\bx')$.
In view of \eqref{uwc},
there exists $T_2>0$ such that
\begin{equation}\label{e14-thm1.1}
\left|\int_{\R^n_+}f(\bx')P(T_2, \bx, d\bx')-\bar f\right|<\eps
\text{ for any } \bx\in K_0
\end{equation}
Since $\BX$ is a Markov-Feller process on $\R^n_+$,
we can find a sufficiently small $\delta=\delta(\eps)>0$ such that
\begin{equation}\label{e15-thm1.1}
\left|\int_{\R^n_+}f(\bx')P(T_2, \bx_1, d\bx')-\int_{\R^n_+}f(\bx')P(T_2, \bx_2, d\bx')\right|<\eps\,\text{
given that }\,\|\bx_1-\bx_2\|\leq \delta.
\end{equation}
Thus, \eqref{e14-thm1.1} and \eqref{e15-thm1.1} imply
\begin{equation}\label{e16-thm1.1}
\left|\int_{\R^n_+}f(\bx')P(T_2, \bx, d\bx')-\bar f\right|<2\eps
\text{ for any } \bx\in K_\delta.
\end{equation}
Since $X_{j^*+1},\dots, X_n$ converges to $0$ almost surely,
there exists $T_3>T_1$ such that
\begin{equation}\label{e17-thm1.1}
\PP_{\bx_0}\left\{X_j(t)\leq \delta, j=j^*+1,\dots,n\right\}>1-\eps\,\text{ for any }\,t\geq T_3.\end{equation}
We deduce from \eqref{e13-thm1.1} and \eqref{e17-thm1.1}
that
\begin{equation}\label{e18-thm1.1}
P(t, \bx_0,K_\delta)=\PP_{\bx_0}\left\{\BX_j(t)\in K_\delta\right\}>1-2\eps\,\text{ for any }\,t\geq T_3.
\end{equation}
 For any $t\geq T_3+T_2$, we have from the Chapman-Kolmogorov equation, \eqref{e16-thm1.1}, \eqref{e18-thm1.1} and $|f(\bx)|\leq 1$ that
 $$
 \begin{aligned}
\left|\int_{\R^n_+}f(\bx')P(t, \bx_0, d\bx')-\bar f\right|=&
\left|\int_{\R^n_+}\left(\int_{\R^n_+}f(\bx')P(T_2, \bx, d\bx')-\bar f\right)P(t-T_2,\bx_0,d\bx)\right|\\
\leq&\left|\int_{K_\delta}\left(\int_{\R^n_+}f(\bx')P(T_2, \bx, d\bx')-\bar f\right)P(t-T_2,\bx_0,d\bx)\right|\\
&+\left|\int_{\R^n_+\setminus K_\delta}\left(\int_{\R^n_+}f(\bx')P(T_2, \bx, d\bx')-\bar f\right)P(t-T_2,\bx_0,d\bx)\right|\\
\leq&2\eps(1-\eps)+2(2\eps)\leq 6\eps,
\end{aligned}
$$
which leads to \eqref{e19-thm1.1}.
The proof is complete.
\end{proof}
\begin{proof} [Proof of Theorem \ref{t:main} (iv)]
If $\Sigma_{j^*}$
is positive definite,
then by Theorem \ref{t:main} part (i) for $\bx\in\R^{(j^*),\circ}_+$ one has that as $t\to\infty$ the transition probability
$P(t, x, \cdot)$ converges in total variation to a unique invariant probability measure $\pi_{j^*}$.
Moreover, the convergence is uniform in each compact set of $\R^{(j^*),\circ}_+$
(due to the property of the Lyapunov function constructed in the proof).
As a result \eqref{uwc} is satisfied and the conclusion follows by part (iii) of Theorem \ref{t:main}.
\end{proof}

\end{document}